\documentclass[graybox]{svmult}

\usepackage{mathptmx}       
\usepackage{helvet}         
\usepackage{courier}        

\usepackage{makeidx}         
\usepackage{multicol}        
\usepackage[bottom]{footmisc}

\usepackage{amsmath,amssymb}
\usepackage[misc,geometry]{ifsym}
\usepackage{cite}

\makeindex

\begin{document}

\title*{The fuzzy Henstock--Kurzweil delta integral\\ 
on time scales}

\author{Dafang Zhao, Guoju Ye, Wei Liu and Delfim F. M. Torres}

\institute{Dafang Zhao \at 
College of Science, Hohai University, Nanjing, 210098, P. R. China \at
School of Mathematics and Statistics, Hubei Normal University,
Huangshi, 435002, P. R. China\\ 
\email{dafangzhao@163.com}
\and Guoju Ye \at 
College of Science, Hohai University, Nanjing, 210098, P. R. China\\
\email{yegj@hhu.edu.cn}
\and Wei Liu \at
College of Science, Hohai University, Nanjing, 210098, P. R. China\\
\email{liuw626@hhu.edu.cn}
\and Delfim F. M. Torres (\Letter) \at
Center for Research and Development in Mathematics and Applications (CIDMA),\\
Department of Mathematics, University of Aveiro, 3810--193 Aveiro, Portugal\\
\email{delfim@ua.pt}
}

\maketitle


\abstract{We investigate  properties of the fuzzy Henstock--Kurzweil delta
integral (shortly, FHK $\Delta$-integral) on time scales, and obtain
two necessary and sufficient conditions for FHK
$\Delta$-integrability. The concept of uniformly FHK
$\Delta$-integrability is introduced. Under this concept, we obtain
a uniformly integrability convergence theorem. Finally, we prove
 monotone and dominated convergence theorems for the FHK
$\Delta$-integral.}


\smallskip

\noindent{\bf Mathematics Subject Classification (2010 )\/}: 26A42; 26E50; 26E70.


\section{Introduction}

The Lebesgue integral, with its convergence properties, is superior
to the Riemann integral. However, a disadvantage with respect to
Lebesgue's integral, is that it is hard to understand without
substantial mathematical maturity.  Also, the Lebesgue integral does
not inherit the naturalness of the Riemann integral. Henstock
\cite{H1} and Kurzweil \cite{K1} gave, independently, a slight, yet
powerful, modification of the Riemann integral to get the now called
Henstock--Kurzweil (HK) integral, which possesses all the convergence
properties of the Lebesgue integral. For the fundamental results of
HK integral, we refer to the papers \cite{BPM,H1,S1,S2,Y,ZY}
and monographs \cite{G,L2,SY}. As an important branch of the
HK integration theory, the fuzzy Henstock--Kurzweil (FHK) 
integral has been extensively studied in 
\cite{BPM12,D16,G04,GS09,HG15,M15,SH13,SH014,WG00,WG01}.

In 1988, Hilger introduced the theory of time scales in his Ph.D.
thesis \cite{H4}. A time scale $\mathbb{T}$ is an arbitrary nonempty
closed subset of $\mathbb{R}$. The aim is to unify and generalize
discrete and continuous dynamical systems, see, e.g.,
\cite{BT16,BCT2,BCT1,BP1,BP2,MR3393831,FB16,G1,OTT}. In \cite{PT},
Peterson and Thompson introduced a more general concept of integral,
i.e., the HK $\Delta$-integral, which gives a common generalization of
the Riemann $\Delta$ and Lebesgue $\Delta$-integral. The theory of HK
integration for real-valued and vector-valued functions on time
scales has been developed rather intensively, see, e.g., the
papers \cite{A,C,FMS,MS2,MS1,ST1,S3,T,MyID:375} and
references cited therein.

In 2015, Fard and Bidgoli introduced the FHK delta integral 
and presented some of its basic properties \cite{FB15}. 
Nonetheless, to our best knowledge, there is no systematic theory 
for the FHK delta integral on time scales. In this work, in order 
to complete the FHK delta integration theory, we give two necessary
and sufficient conditions of FHK delta integrability 
(see Theorems~\ref{thm3.3} and~\ref{thm3.8}). Moreover, 
we obtain some convergence theorems
for the FHK delta integral, in particular
Theorem~\ref{thm3.10} of dominated convergence
and Theorem~\ref{thm3.11} of monotone convergence.

After Section~\ref{sec:2} of preliminaries, 
in Section~\ref{sec:3} the definition of FHK delta integral
is introduced, and our necessary and sufficient conditions of FHK delta integrability
proved. We also obtain some convergence theorems. Finally, in Section~\ref{sec:4},
we give conclusions and point out some directions that deserve further study.


\section{Preliminaries }
\label{sec:2}

A fuzzy subset of the real axis $u:\mathbb{R}\rightarrow[0,1]$ 
is called a fuzzy number provided that 
\begin{enumerate}
\item[(1)] $u$ is normal: there exists $x_{0}\in\mathbb{R}$ with $u(x_{0})=1$;

\item[(2)] $u$ is fuzzy convex: 
$u(\lambda x_{1}+(1-\lambda)x_{2})\geq \min\{u(x_{1}),u(x_{2})\}$
for all $x_{1},x_{2}\in\mathbb{R}$ and all $\lambda \in(0,1)$;

\item[(3)] $u$ is upper semi-continuous;

\item[(4)] $[u]^{0}=\overline{\{x\in\mathbb{R}:u(x)>0\}}$ is compact.
\end{enumerate}
Denote by $\mathbb{R}_{\mathcal{F}}$ the space of fuzzy numbers.
 We define the $\alpha$-level set $[u]^{\alpha}$ by
$$
[u]^{\alpha}=\left\{x\in\mathbb{R}:u(x)\geq\alpha\right\},\quad
\alpha\in(0,1].
$$
From conditions $(1)$--$(4)$, $[u]^{\alpha}$ is denoted by
$[u]^{\alpha}=\left[\underline{u^{\alpha}},\overline{u^{\alpha}}\right]$.
For $u_{1},u_{2}\in\mathbb{R}_{\mathcal{F}}$ and
$\lambda\in\mathbb{R}$, we define
$$[u_{1}+u_{2}]^{\alpha}=[u_{1}]^{\alpha}+[u_{2}]^{\alpha} \ \ {\rm and}\ \
[\lambda\odot u_{1}]^{\alpha}=\lambda[u_{1}]^{\alpha}$$
for all $\alpha\in[0,1]$.
The Hausdorff distance between $u_{1}$ and $u_{2}$ is defined by
$$
\mathbf{D}(u_{1},u_{2})=\sup_{\alpha\in[0,1]}\max\left\{|\underline{u_{1}^{\alpha}}
-\underline{u_{2}^{\alpha}}|,|\overline{u_{1}^{\alpha}}-\overline{u_{2}^{\alpha}}|\right\}.
$$
Then, the metric space $(\mathbb{R}_{\mathcal{F}},\mathbf{D})$ is complete.
Let $a,b\in \mathbb{T}$. We define the half-open interval $[a,b)_{\mathbb{T}}$ by
$$
[a,b)_{\mathbb{T}}
=\left\{x\in \mathbb{T}: a \leq x < b\right\}.
$$
The open and closed intervals are defined similarly. For $x\in \mathbb{T}$, 
we denote by $\sigma$ the forward jump operator, i.e., $\sigma(x):=\inf\{y>x: y\in \mathbb{T}\}$, 
and by $\rho$ the backward jump operator, i.e., $\rho(x):=\sup\{y<x: y\in \mathbb{T}\}$. 
Here, we put $\sigma(\sup\mathbb{T})=\sup\mathbb{T}$ and $\rho(\inf\mathbb{T})=\inf\mathbb{T}$, 
where $\sup\mathbb{T}$ and $\inf\mathbb{T}$ are finite. In this situation, 
$\mathbb{T}^{\kappa}:=\mathbb{T}\backslash \{\sup\mathbb{T}\}$
and $\mathbb{T}_{\kappa}:=\mathbb{T}\backslash\{\inf\mathbb{T}\}$,
otherwise, $\mathbb{T}^{\kappa}:=\mathbb{T}$ and $\mathbb{T}_{\kappa}:=\mathbb{T}$. 
If $\sigma(x)>x$, then we say that $x$ is right-scattered, 
while if $\rho(x)<x$, then we say that $x$ is left-scattered. 
If $\sigma(x)=x$ and $x< \sup\mathbb{T}$, then $x$ is called right-dense, 
and if $\rho(x)=x$ and $x> \inf\mathbb{T}$, then $x$ is left-dense. 
The graininess functions $\mu$ and $\eta$ are defined by $\mu(x):=\sigma(x)-x$ 
and $\eta(x):=x-\rho(x)$, respectively.

In what follows, all considered intervals are intervals in
$\mathbb{T}$. A division $D$ of $[a,b]_{\mathbb{T}}$ is a finite set
of interval-point pairs
${\{([x_{i-1},x_{i}]_{\mathbb{T}},\xi_{i})\}}^{n}_{i=1}$ such that
$$
\bigcup^{n}_{i=1}[x_{i-1},x_{i}]_{\mathbb{T}}=[a,b]_{\mathbb{T}}
$$
and $\xi_{i}\in[a,b]_{\mathbb{T}}$ for each $i$. We write $\Delta
x_{i}=x_{i}-x_{i-1}$. We say that
$$
\delta(\xi)=(\delta_{L}(\xi),\delta_{R}(\xi))
$$ 
is a $\Delta$-gauge on $[a,b]_{\mathbb{T}}$ if $\delta_{L}(\xi)>0$ on
$(a,b]_{\mathbb{T}}$, $\delta_{R}(\xi)>0$ on $[a,b)_{\mathbb{T}}$,
$\delta_{L}(a)\geq 0$, $\delta_{R}(b)\geq 0$ and $\delta_{R}(\xi)\geq
\mu(\xi)$ for any $\xi\in [a,b)_{\mathbb{T}}$. The symbol
$\Gamma(\Delta, [a,b]_{\mathbb T})$ stands for the set of
$\Delta$-gauge on $[a,b]_{\mathbb T}$. Let $\delta^{1}(\xi)$ and
$\delta^{2}(\xi)$ be $\Delta$-gauges such that 
$$
0<\delta^{1}_{L}(\xi)< \delta^{2}_{L}(\xi)
$$ 
for any $\xi\in (a,b]_{\mathbb{T}}$ and 
$0< \delta^{1}_{R}(\xi)<\delta^{2}_{R}(\xi)$ 
for any $\xi\in [a,b)_{\mathbb{T}}$. Then we
call $\delta^{1}(\xi)$ finer than $\delta^{2}(\xi)$ and
write $\delta^{1}(\xi)<\delta^{2}(\xi)$. We say that
$D={\{([x_{i-1},x_{i}]_{\mathbb{T}},\xi_{i})\}}^{n}_{i=1}$ is a
$\delta$-fine HK division of $[a,b]_{\mathbb{T}}$ if
$\xi_{i}\in[x_{i-1},x_{i}]_{\mathbb{T}}\subset(\xi_{i}
-\delta_{L}(\xi_{i}),\xi_{i}+\delta_{R}(\xi_{i}))_{\mathbb{T}}$ 
for each $i$. Let $\mathfrak{D}(\delta,[a,b]_{\mathbb{T}})$ 
be the set of all $\delta$-fine HK divisions of
$[a,b]_{\mathbb{T}}$. Given an arbitrary $D\in
\mathfrak{D}(\delta,[a,b]_{\mathbb{T}})$,
$D={\{([x_{i-1},x_{i}]_{\mathbb{T}},\xi_{i})\}}^{n}_{i=1}$, we write
$$
S(f,D,\delta)=\displaystyle \sum^{n}_{i=1}f(\xi_{i})\Delta x_{i}
$$
for integral sums over $D$, whenever
$f:[a,b]_{\mathbb{T}}\rightarrow \mathbb{R}_{\mathcal{F}}$.

\begin{lemma}[See \cite{K87}]
\label{lem2.1}
Suppose that $u\in \mathbb{R}_{\mathcal{F}}$. Then,
\begin{enumerate}
\item[(1)] the interval $[u]^{\alpha}$ is closed for $\alpha\in [0,1]$;

\item[(2)] $[u]^{\alpha_{1}}\supset [u]^{\alpha_{2}}$
for $0\leq \alpha_{1}\leq \alpha_{2}\leq1$;

\item[(3)] for any sequence $\{\alpha_{n}\}$ satisfying
$\alpha_n\leq \alpha_{n+1}$ and   $\alpha_{n}\to\alpha\in (0,1]$, 
we have $\bigcap_{n=1}^{\infty}[u]^{\alpha_{n}}=[u]^{\alpha}$.
\end{enumerate}
Conversely, if a collection of subsets $\{\mathrm{u}^{\alpha}:\alpha\in [0,1]\}$ 
verify (1)--(3), then there exists a unique $\mathrm{u}\in \mathbb{R}_{\mathcal{F}}$
such that $[u]^{\alpha}=\mathrm{u}^{\alpha}$ for $\alpha\in (0,1]$ and
$[u]^{0}=\overline{\bigcup_{\alpha\in (0,1]}\mathrm{u}^{\alpha}}\subset \mathrm{u}^{0}$.
\end{lemma}

\begin{lemma}[See \cite{GV86}]
\label{lem2.2}
Suppose that $u\in \mathbb{R}_{\mathcal{F}}$. Then,
\begin{enumerate}
\item[(1)] $\underline{u^{\alpha}}$ is bounded and nondecreasing;

\item[(2)] $\overline{u^{\alpha}}$ is bounded and nonincreasing;

\item[(3)] $\underline{u^{1}}\leq \overline{u^{1}}$;

\item[(4)] for $c\in (0,1]$, $\lim_{\alpha\rightarrow c^{-}}
\underline{u^{\alpha}}=\underline{u^{c}}$ and
$\lim_{\alpha\rightarrow c^{-}}\overline{u^{\alpha}}=\overline{u^{c}}$;

\item[(5)] $\lim_{\alpha\rightarrow 0^{+}}\underline{u^{\alpha}}
=\underline{u^{0}}$ and $\lim_{\alpha\rightarrow 0^{+}}\overline{u^{\alpha}}=\overline{u^{0}}$.
\end{enumerate}
Conversely, if $\underline{u^{\alpha}}$ and $\overline{u^{\alpha}}$
satisfy items (1)--(5), then there exists 
$\mathrm{u}\in \mathbb{R}_{\mathcal{F}}$ such that
$$
[\mathrm{u}]^{\alpha}=\left[\underline{\mathrm{u}^{\alpha}},
\overline{\mathrm{u}^{\alpha}}\right]
=\left[\underline{u^{\alpha}},
\overline{u^{\alpha}}\right].
$$ 
\end{lemma}


\section{The fuzzy Henstock--Kurzweil delta integral}
\label{sec:3}

We introduce the concept of fuzzy Henstock--Kurzweil 
(FHK) delta integrability.

\begin{definition}
\label{defn3.1} 
A function $f:[a,b]_{\mathbb{T}}\rightarrow
\mathbb{R}_{\mathcal{F}}$ is called FHK 
$\Delta$-integrable on $[a,b]_{\mathbb{T}}$ with the FHK
$\Delta$-integral $\tilde{A}=(FHK)\displaystyle
\int_{[a,b]_{\mathbb{T}}}f(x)\Delta x$, if for each $\epsilon>0$
there exists a $\delta\in \Gamma(\Delta, [a,b]_{\mathbb{T}})$ such that
$\mathbf{D}\left(S(f,D,\delta), \tilde{A}\right)<\epsilon$
for each $D\in \mathfrak{D}(\delta,[a,b]_{\mathbb{T}})$. 
The family of all FHK $\Delta$-integrable functions on 
$[a,b]_{\mathbb{T}}$ is denoted by $\mathcal{FHK}_{[a,b]_{\mathbb{T}}}$.
\end{definition}

\begin{remark}
\label{rmk3.1}
It is clear that Definition~\ref{defn3.1} is more general than the HK $\Delta$-integral
introduced by Peterson and Thompson in \cite{PT} and more general than the
FH integral introduced by Wu and Gong in \cite{WG00,WG01}.
\end{remark}

The proofs of Theorems~\ref{thm3.1} and \ref{thm3.2} are straightforward
and are left to the reader.

\begin{theorem}
\label{thm3.1}
The FHK $\Delta$-integral of $f(x)$ is unique.
\end{theorem}

\begin{theorem}
\label{thm3.2}
If $f(x), g(x)\in\mathcal{FHK}_{[a,b]_{\mathbb{T}}}$
and $\alpha,\beta\in \mathbb{R}$, then 
$$\alpha f(x)
+\beta g(x)\in\mathcal{FHK}_{[a,b]_{\mathbb{T}}}
$$ 
with
\begin{multline*}
(FHK)\int_{[a,b]_{\mathbb{T}}}(\alpha f(x)+\beta g(x))\Delta x\\
=\alpha(FHK)\int_{[a,b]_{\mathbb{T}}}f(x)\Delta x
+\beta(FHK)\int_{[a,b]_{\mathbb{T}}}g(x)\Delta x.
\end{multline*}
\end{theorem}

Follows a Cauchy--Bolzano condition for the FHK $\Delta$-integral.

\begin{theorem}[The Cauchy--Bolzano condition] 
\label{thm3.3} 
Function $f(x)\in\mathcal{FHK}_{[a,b]_{\mathbb{T}}}$ 
if and only if for each $\epsilon>0$ there exists a 
$\delta\in \Gamma(\Delta,[a,b]_{\mathbb{T}})$ such that
$$
\mathbf{D}\left(S(f,D_{1},\delta),
S(f,D_{2},\delta)\right)<\epsilon
$$
for any $D_{1}, D_{2}\in \mathfrak{D}(\delta,[a,b]_{\mathbb{T}})$.
\end{theorem}

\begin{proof}
(Necessity) Let $\epsilon>0$. By hypothesis, there exists 
$\delta\in \Gamma(\Delta, [a,b]_{\mathbb{T}})$ such that
$$
\mathbf{D}\left(S(f,D,\delta), (FHK)\int_{[a,b]_{\mathbb{T}}}f(x)\Delta x\right)
<\frac{\epsilon}{2}
$$
for any $D\in \mathfrak{D}(\delta,[a,b]_{\mathbb{T}})$.
Let $D_{1}, D_{2}\in \mathfrak{D}(\delta,[a,b]_{\mathbb{T}})$. Then,
\begin{equation*}
\begin{split}
&\mathbf{D}\left(S(f,D_{1},\delta), S(f,D_{2},\delta)\right)\\
&\leq \mathbf{D}\left(S(f,D_{1},\delta), (FHK)\int_{[a,b]_{\mathbb{T}}}f(x)\Delta x\right)
+\mathbf{D}\left(S(f,D_{2},\delta), (FHK)\int_{[a,b]_{\mathbb{T}}}f(x)\Delta x\right)\\
&<\frac{\epsilon}{2}+\frac{\epsilon}{2}=\epsilon.
\end{split}
\end{equation*}
(Sufficiency) For each $n$, choose a $\delta_{n}\in \Gamma(\Delta_{n},
[a,b]_{\mathbb{T}})$ such that
$$
\mathbf{D}\left(S(f,D_{1},\delta_{n}), S(f,D_{2},\delta_{n})\right)
<\frac{1}{n}
$$
for any $D_{1}, D_{2}\in \mathfrak{D}(\delta_{n},[a,b]_{\mathbb{T}})$. 
Replacing $\delta_{n}$ by $\bigcap_{j=1}^{n}\delta_{j}=\delta_{n}$,
we may assume that $\delta_{n+1}\subset \delta_{n}$. For each $n$,
fix a $D_{n}\in \mathfrak{D}(\delta_{n},[a,b]_{\mathbb{T}})$. For
$j>n$, we have $\delta_{j}\subset \delta_{n}$, so $D_{j}\in
\mathfrak{D}(\delta_{n},[a,b]_{\mathbb{T}})$. Thus,
$\mathbf{D}\left(S(f,D_{n},\delta_{n}),
S(f,D_{j},\delta_{n})\right)<\frac{1}{n}$
and it follows that $\{S(f,D_{n},\delta_{n})\}$ is a Cauchy
sequence. We denote the limit of $\{S(f,D_{n},\delta_{n})\}$ by
$\tilde{A}$ and let $\epsilon>0$. Choose $N>\frac{2}{\epsilon}$ and
let $D\in \mathfrak{D}(\delta_{N},[a,b]_{\mathbb{T}})$. Then,
\begin{equation*}
\begin{split}
\mathbf{D}\left(S(f,D,\delta_{N}), \tilde{A}\right)
&\leq \mathbf{D}\left(S(f,D,\delta_{N}),
S(f,D_{N},\delta_{N})\right)
+\mathbf{D}\left(S(f,D_{N},\delta_{N}), \tilde{A}\right)\\
&<\frac{1}{N}+\frac{1}{N}\\
&<\epsilon.
\end{split}
\end{equation*}
Hence, $f(x)\in\mathcal{FHK}_{[a,b]_{\mathbb{T}}}$.
\end{proof}

\begin{theorem}
\label{thm3.4}
Let $c\in (a,b)_{\mathbb{T}}$. If 
$f(x)\in\mathcal{FHK}_{[a,c]_{\mathbb{T}}}
\bigcap\mathcal{FHK}_{[c,b]_{\mathbb{T}}}$, then 
$$
f(x)\in\mathcal{FHK}_{[a,b]_{\mathbb{T}}}
$$ 
with
$$
(FHK)\int_{[a,b]_{\mathbb{T}}}f(x)\Delta x
=(FHK)\int_{[a,c]_{\mathbb{T}}}f(x)\Delta x
+(FHK)\int_{[c,b]_{\mathbb{T}}}f(x)\Delta x.
$$
\end{theorem}

\begin{proof}
Let $\epsilon>0$. By assumption, there exist $\Delta$-gauges
$$
\delta^{i}(\xi)=(\delta^{i}_{L}(\xi),\delta^{i}_{R}(\xi)),
\quad i = 1, 2, 
$$
such that
\begin{equation*}
\begin{split}
&\mathbf{D}\left(S(f,D_{1},\delta^{1}),
\  (FHK)\int_{[a,c]_{\mathbb{T}}}f(x)\Delta x\right)<\epsilon,\\
&\mathbf{D}\left(S(f,D_{2},\delta^{2}),
\ (FHK)\int_{[c,b]_{\mathbb{T}}}f(x)\Delta x\right)<\epsilon,
\end{split}
\end{equation*}
respectively for any $D_{1}\in \mathfrak{D}(\delta^{1},[a,c]_{\mathbb{T}})$, $D_{1}={\{([x^{1}_{k-1},x^{1}_{\kappa}]_{\mathbb{T}},\xi^{1}_{\kappa})\}}^{n}_{k=1}$,
and for any $D_{2}\in \mathfrak{D}(\delta^{2},[c,b]_{\mathbb{T}})$,
$D_{2}={\{([x^{2}_{k-1},x^{2}_{\kappa}]_{\mathbb{T}},\xi^{2}_{\kappa})\}}^{m}_{k=1}$. 
We define $\delta(\xi)=(\delta_{L}(\xi),\delta_{R}(\xi))$
on $[a,b]_{\mathbb{T}}$ by setting
\begin{equation*}
\delta_{L}(\xi)=
\begin{cases}
\delta^{1}_{L}(\xi),& \text{if $\xi\in [a,c)_{\mathbb{T}}$},\\
\delta^{1}_{L}(\xi),& \text{if $\xi=c=\rho(c)$},\\
\min\left\{\delta^{1}_{L}(\xi),\frac{\eta(c)}{2}\right\},& \text{if $\xi=c>\rho(c)$},\\
\min\left\{\delta^{2}_{L}(\xi),\frac{\xi-c}{2}\right\},& \text{if $\xi\in (c,b]_{\mathbb{T}}$},
\end{cases}
\end{equation*}
and
\begin{equation*}
\delta_{R}(\xi)=
\begin{cases}
\min\left\{\delta^{1}_{R}(\xi),
\max\left\{\mu(\xi),\frac{c-\xi}{2}\right\}\right\},
& \text{if $\xi\in [a,c)_{\mathbb{T}}$},\\
\min\{\delta^{2}_{R}(\xi)\},
& \text{if $\xi\in [c,b]_{\mathbb{T}}$}.
\end{cases}
\end{equation*}
Now, let $D\in \mathfrak{D}(\delta,[a,b]_{\mathbb{T}})$, 
$D={\{([x_{k-1},x_{k}]_{\mathbb{T}},\xi_{k})\}}^{p}_{k=1}$. 
It follows that
\begin{itemize}
\item[(i)] either $c=\xi_{q}$ and $t_{q}>c$;

\item[(ii)] \ or $\xi_{q}=\rho(c)<c$ and $t_{q}=c$. 
\end{itemize}
The case (ii) is straightforward. For (i), one has
\begin{equation*}
\begin{split}
\mathbf{D}&\left(S(f,D,\delta),\  (FHK)\int_{[a,c]_{\mathbb{T}}}
f(x)\Delta x+(FHK)\int_{[c,b]_{\mathbb{T}}}f(x)\Delta x\right)\\
&=\mathbf{D}\left(\sum_{k=1}^{p}f(\xi_{k})\Delta x_{k},
\  (FHK)\int_{[a,c]_{\mathbb{T}}}
f(x)\Delta x+(FHK)\int_{[c,b]_{\mathbb{T}}}f(x)\Delta x\right)\\
&\leq \mathbf{D}\left(\sum_{k=1}^{q-1}f(\xi_{k})\Delta x_{k}+f(c)(c-t_{q-1}),
\  (FHK)\int_{[a,c]_{\mathbb{T}}}
f(x)\Delta x\right)\\
&\quad +\mathbf{D}\left(\sum_{k=q+1}^{p}f(\xi_{k})\Delta x_{k}+f(c)(t_{q}-c),
\  (FHK)\int_{[c,b]_{\mathbb{T}}}
f(x)\Delta x\right)\\
&<\epsilon+\epsilon=2\epsilon.
\end{split}
\end{equation*}
The intended result follows.
\end{proof}

\begin{corollary}
\label{thm3.5} 
If $f\in\mathcal{FHK}_{[a,b]_{\mathbb{T}}}$, then
$f\in\mathcal{FHK}_{[r,s]_{\mathbb{T}}}$ for any
$[r,s]_{\mathbb{T}}\subset[a,b]_{\mathbb{T}}$.
\end{corollary}

\begin{definition}[See \cite{PT}]
\label{defn3.2} 
Let $\mathbb{T'}\subset\mathbb{T}$. We say
$\mathbb{T'}$ has delta measure zero if it has Lebesgue measure zero
and contains no right-scattered points. A property $\mathcal{P}$ is
said to hold $\Delta$ a.e. on $\mathbb{T}$ if there exists
$\mathbb{T'}$ of measure zero such that $\mathcal{P}$ holds for
every $t\in \mathbb{T} \setminus \mathbb{T'}$.
\end{definition}

\begin{theorem}
\label{thm3.6}
Let $f(x)=g(x)$ $\Delta$ a.e. on $[a,b]_{\mathbb{T}}$. 
If $f(x)\in\mathcal{FHK}_{[a,b]_{\mathbb{T}}}$,
then so $g(x)$. Moreover,
$$
(FHK)\int_{[a,b]_{\mathbb{T}}}f(x)\Delta x
=(FHK)\int_{[a,b]_{\mathbb{T}}}g(x)\Delta x.
$$
\end{theorem}

\begin{proof}
Let $\epsilon>0$. Then there exists a 
$\delta\in \Gamma(\Delta, [a,b]_{\mathbb{T}})$ such that
$$
\mathbf{D}\left(S(f,D,\delta), 
(FHK)\int_{[a,b]_{\mathbb{T}}}f(x)\Delta x\right)<\epsilon
$$
for any $D\in \mathfrak{D}(\delta,[a,b]_{\mathbb{T}})$, 
$D={\{([x_{i-1},x_{i}]_{\mathbb{T}},\xi_{i})\}}$. 
Set $E=\sum_{j=1}^{\infty}E_{j}$, where
$$
E_{j}=\Big\{x:j-1<\mathbf{D}(f(x),g(x))\leq j,\
 \ t\in[a,b]_{\mathbb{T}}\Big\}_{j=1}^{\infty}.
$$
For each $j$, there exists  $F_{j}$ consisting of a collection of
open intervals  with total length less than $\epsilon
\cdot2^{-j}\cdot j^{-1}$, such that $E_{j}\subset F_{j}$. Define
\begin{equation*}
\delta(\xi)=
\begin{cases}
(\delta^{0}_{L}(\xi),\delta^{0}_{R}(\xi)),
& \text{ if } \xi\in [a,b]_{\mathbb{T}}\backslash E,\\
(\delta^{1}_{L}(\xi),\delta^{1}_{R}(\xi)),
& \text{ if }  \xi\in E_{j} \text{ satisfies } (\xi-\delta^{1}_{L}(\xi),\xi
+\delta^{1}_{R}(\xi))_{\mathbb{T}}\subset F_{j}.
\end{cases}
\end{equation*}
Then, for any $D\in \mathfrak{D}(\delta,[a,b]_{\mathbb{T}})$,
$D={\{([x_{i-1},x_{i}]_{\mathbb{T}},\xi_{i})\}}$, one has
\begin{equation*}
\begin{split}
\mathbf{D}&\left(S(g,D,\delta), \ (FHK)\int_{[a,b]_{\mathbb{T}}}f(x)\Delta x\right)\\
&= \mathbf{D}\left(\sum_{\xi_{i}\in [a,b]_{\mathbb{T}}}g(\xi_{i})\Delta x_{i},
\ (FHK)\int_{[a,b]_{\mathbb{T}}}f(x)\Delta x\right) \\
&= \mathbf{D}\left(\sum_{\xi_{i}\in E}g(\xi_{i})\Delta x_{i}
+\sum_{\xi_{i}\in [a,b]_{\mathbb{T}}
\backslash E}g(\xi_{i})\Delta x_{i}, \ (FHK)\int_{[a,b]_{\mathbb{T}}}f(x)\Delta x\right)\\
&= \mathbf{D}\left(\sum_{\xi_{i}\in E}g(\xi_{i})\Delta x_{i}
+\sum_{\xi_{i}\in [a,b]_{\mathbb{T}}\backslash E}f(\xi_{i})\Delta x_{i}
+\sum_{\xi_{i}\in E}f(\xi_{i})\Delta x_{i},\right.\\
&\qquad \left.(FHK)\int_{[a,b]_{\mathbb{T}}}f(x)\Delta x
+\sum_{\xi_{i}\in E}f(\xi_{i})\Delta x_{i}\right).
\end{split}
\end{equation*}
Therefore,
\begin{equation*}
\begin{split}
\mathbf{D}&\left(S(g,D,\delta), \ (FHK)\int_{[a,b]_{\mathbb{T}}}f(x)\Delta x\right)\\
&\leq \mathbf{D}\left(\sum_{\xi_{i}\in [a,b]_{\mathbb{T}}}
f(\xi_{i})\Delta x_{i}, \ (FHK)\int_{[a,b]_{\mathbb{T}}}f(x)\Delta x\right)\\
&\quad +\mathbf{D}\left(\sum_{\xi_{i}\in E}g(\xi_{i})\Delta x_{i},
\ \sum_{\xi_{i}\in E}f(\xi_{i})\Delta x_{i}\right)\\
&\leq \epsilon+\sum^{\infty}_{j=1}\sum_{\xi_{i}\in E_{j}}
\mathbf{D}\left(f(\xi_{i}), g(\xi_{i})\right)\Delta x_{i}\\
&\leq 2\epsilon.
\end{split}
\end{equation*}
The proof is complete.
\end{proof}

\begin{theorem}[See \cite{PT}]
\label{thm3.7} 
Let $[a,b]_{\mathbb{T}}$ be given. Assume
\begin{enumerate}
\item[(1)] $\lim _{n\rightarrow \infty}f_{n}(x)= f(x)$ holds $\Delta$ a.e.;

\item[(2)] $G(x)\leq f_{n}(x)\leq H(x)$ holds $\Delta$ a.e.;

\item[(3)] $f_{n}(x),\ G(x),\ H(x)\in \mathcal{HK}_{[a,b]_{\mathbb{T}}}$.
\end{enumerate}
Then $f(x)\in \mathcal{HK}_{[a,b]_{\mathbb{T}}}$. Moreover,
$$
\lim_{n\rightarrow\infty}(HK)\int_{[a,b]_{\mathbb{T}}}f_{n}(x)\Delta x
=(HK)\int_{[a,b]_{\mathbb{T}}}f(x)\Delta x.
$$
\end{theorem}

\begin{theorem}
\label{thm3.8} 
Function $f(x)\in\mathcal{FHK}_{[a,b]_{\mathbb{T}}}$ if and only if
$\underline{f(x)^{\alpha}}$, $\overline{f(x)^{\alpha}} \in
\mathcal{HK}_{[a,b]_{\mathbb{T}}}$ for all $\alpha\in[0,1]$
uniformly, i.e., the $\Delta$-gauge in Definition~\ref{defn3.1} is
independent of $\alpha$.
\end{theorem}

\begin{proof}
(Necessity) Let $\tilde{A}=(FHK)\displaystyle \int_{[a,b]_{\mathbb{T}}}f(x)\Delta x$. 
Given $\epsilon>0$, there exists a 
$$
\delta\in \Gamma(\Delta,
[a,b]_{\mathbb{T}})
$$ 
such that $\mathbf{D}\left(S(f,D,\delta),
\tilde{A}\right)<\epsilon$ for any $D\in
\mathfrak{D}(\delta,[a,b]_{\mathbb{T}})$. Then,
\begin{equation*}
\begin{split}
\sup_{\alpha\in[0,1]}&\max\left\{\left|\underline{\left[
S(f,D,\delta)\right]^{\alpha}}
-\underline{\tilde{A}^{\alpha}}\right|,\left|\overline{\left[S(f,D,\delta)\right]^{\alpha}}
-\overline{\tilde{A}^{\alpha}}\right|\right\}\\
&=\sup_{\alpha\in[0,1]}\max\left\{\left|S(\underline{f^{\alpha}},D,\delta)
-\underline{\tilde{A}^{\alpha}}\right|, \left|S(\overline{f^{\alpha}},D,\delta)
-\overline{\tilde{A}^{\alpha}}\right|\right\} \\
&<\epsilon
\end{split}
\end{equation*}
and
$$
\left|S(\underline{f^{\alpha}},D,\delta)
-\underline{\tilde{A}^{\alpha}}\right|<\epsilon,
\quad
\left|S(\overline{f^{\alpha}},D,\delta)
-\overline{\tilde{A}^{\alpha}}\right|<\epsilon
$$
for any $\alpha\in[0,1]$ and for any $D\in \mathfrak{D}(\delta,[a,b]_{\mathbb{T}})$.
Thus, $\underline{f(x)^{\alpha}}$,
$\overline{f(x)^{\alpha}}\in \mathcal{HK}_{[a,b]_{\mathbb{T}}}$ uniformly
for any $\alpha\in[0,1]$.\\

\noindent (Sufficiency) Let $\epsilon>0$. By assumption, there exists a
$\delta\in \Gamma(\Delta, [a,b]_{\mathbb{T}})$ such that
$$
\left|S(\underline{f^{\alpha}},D,\delta)-\underline{\tilde{A}^{\alpha}}\right|<\epsilon,
\quad
\left|S(\overline{f^{\alpha}},D,\delta)
-\overline{\tilde{A}^{\alpha}}\right|<\epsilon
$$
for any $D\in \mathfrak{D}(\delta,[a,b]_{\mathbb{T}})$ and for any
$\alpha\in[0,1]$, where
$$
\underline{\tilde{A}^{\alpha}}=(FHK)\int_{[a,b]_{\mathbb{T}}}\underline{f^{\alpha}}\Delta x,\ \
\overline{\tilde{A}^{\alpha}}=(FHK)\int_{[a,b]_{\mathbb{T}}}\overline{f^{\alpha}}\Delta x.
$$ 
To prove that $\left\{\left[\underline{\tilde{A}^{\alpha}},
\overline{\tilde{A}^{\alpha}}\right], \alpha\in[0,1]\right\}$
represents a fuzzy number, it is enough to check that
$\left[\underline{\tilde{A}^{\alpha}}, \overline{\tilde{A}^{\alpha}}\right]$ 
satisfies items (1)--(3) of Lemma~\ref{lem2.1}:
\begin{enumerate}
\item[(1)] for $\alpha\in[0,1]$, if $\underline{f^{\alpha}}
\leq \overline{f^{\alpha}}$, then
$\underline{\tilde{A}^{\alpha}}\leq\overline{\tilde{A}^{\alpha}}$,
i.e., the interval $\left[\underline{\tilde{A}^{\alpha}},
\overline{\tilde{A}^{\alpha}}\right]$ is closed.

\item[(2)] $\underline{f^{\alpha}}$ and $\overline{f^{\alpha}}$ 
are, respectively, nondecreasing and nonincreasing functions on $[0,1]$.
For any $0\leq \alpha_{1}\leq \alpha_{2}\leq 1$ one has
\begin{eqnarray*}
(FHK)\int_{[a,b]_{\mathbb{T}}}\underline{f^{\alpha_{1}}}\Delta x
&\leq& (FHK)\int_{[a,b]_{\mathbb{T}}}\underline{f^{\alpha_{2}}}\Delta x\\
&\leq&(FHK)\int_{[a,b]_{\mathbb{T}}}\overline{f^{\alpha_{2}}}\Delta x\\
&\leq& (FHK)\int_{[a,b]_{\mathbb{T}}}\overline{f^{\alpha_{1}}}\Delta x.
\end{eqnarray*}
This implies $\left[\underline{\tilde{A}^{\alpha_{1}}},
\overline{\tilde{A}^{\alpha_{1}}}\right]\supset
\left[\underline{\tilde{A}^{\alpha_{2}}},
\overline{\tilde{A}^{\alpha_{2}}}\right]$.

\item[(3)] For any $\{\alpha_{n}\}$ satisfying
$\alpha_n\leq \alpha_{n+1}$ and   $\alpha_{n}\to\alpha\in (0,1]$,
we have 
$$
\bigcap_{n=1}^{\infty}\left[f\right]^{\alpha_{n}}
=\left[f\right]^{\alpha}, 
$$
that is,
$$
\bigcap_{n=1}^{\infty}\left[\underline{f^{\alpha_{n}}},
\overline{f^{\alpha_{n}}}\right]
=\left[\underline{f^{\alpha}}, \overline{f^{\alpha}}\right],
$$
$\lim_{n\rightarrow\infty}\underline{f^{\alpha_{n}}}
=\underline{f^{\alpha}}$ and
$\lim_{n\rightarrow\infty}\overline{f^{\alpha_{n}}}
=\overline{f^{\alpha}}$. Moreover,
$$
\underline{f^{0}}\leq \underline{f^{\alpha_{n}}}
\leq \underline{f^{1}}, 
\quad \overline{f^{1}}
\leq \overline{f^{\alpha_{n}}}\leq \overline{f^{0}}.
$$
Thanks to Theorem~\ref{thm3.7}, we have $\underline{f^{\alpha}},\
\overline{f^{\alpha}}\in\mathcal{HK}_{[a,b]_{\mathbb{T}}}$ and
\begin{gather*}
\lim_{n\rightarrow\infty}(HK)\int_{[a,b]_{\mathbb{T}}}\underline{f^{\alpha_{n}}}\Delta x
=(HK)\int_{[a,b]_{\mathbb{T}}}\underline{f^{\alpha}}\Delta x,\\
\lim_{n\rightarrow\infty}(HK)\int_{[a,b]_{\mathbb{T}}}\overline{f^{\alpha_{n}}}\Delta x
=(HK)\int_{[a,b]_{\mathbb{T}}}\overline{f^{\alpha}}\Delta x.
\end{gather*}
Consequently,
$$
\bigcap_{n=1}^{\infty}\left[\underline{\tilde{A}^{\alpha_{n}}},
\overline{\tilde{A}^{\alpha_{n}}}\right]
=\left[\underline{\tilde{A}^{\alpha}},
\overline{\tilde{A}^{\alpha}}\right].
$$
\end{enumerate}
Define $\tilde{A}$ by $\left\{\left[\underline{\tilde{A}^{\alpha}},
\overline{\tilde{A}^{\alpha}}\right], \alpha\in[0,1]\right\}$.
Thus, 
$$\mathbf{D}\left(S(f,D,\delta), \tilde{A}\right)<\epsilon
$$ 
for each $D\in \mathfrak{D}(\delta,[a,b]_{\mathbb{T}})$.
\end{proof}

\begin{definition}
\label{defn3.3} 
A sequence $\{f_{n}(x)\}$ of HK $\Delta$-integrable
functions is called uniformly FHK $\Delta$-integrable on
$[a,b]_{\mathbb{T}}$ if for each $\epsilon>0$ there exists a
$\delta\in \Gamma(\Delta, [a,b]_{\mathbb{T}})$ such that
$$
\mathbf{D}\left(S(f_{n},D,\delta),\  
(FHK)\int_{[a,b]_{\mathbb{T}}}f_{n}(x)\Delta x\right)
<\epsilon$$ for any 
$D\in \mathfrak{D}(\delta,[a,b]_{\mathbb{T}})$ and for any $n$.
\end{definition}

\begin{theorem}
\label{thm3.9}
Let $f_{n}(x)\in\mathcal{FHK}_{[a,b]_{\mathbb{T}}}$, $n=1,2,\ldots$, satisfy:
\begin{enumerate}
\item[(1)] $\lim_{n\rightarrow\infty}f_{n}(x)= f(x)$  on $[a,b]_{\mathbb{T}}$;

\item[(2)] $f_{n}(x)$ are uniformly FHK $\Delta$-integrable on $[a,b]_{\mathbb{T}}$.
\end{enumerate}
Then $f(x)\in\mathcal{FHK}_{[a,b]_{\mathbb{T}}}$ and
$$
\lim_{n\rightarrow\infty}(FHK)\int_{[a,b]_{\mathbb{T}}}f_{n}(x)\Delta x
=(FHK)\int_{[a,b]_{\mathbb{T}}}f(x)\Delta x.
$$
\end{theorem}

\begin{proof}
Let $\epsilon>0$. By assumption, there exists a $\delta\in
\Gamma(\Delta, [a,b]_{\mathbb{T}})$ such that
$$
\mathbf{D}\left(S(f_{n},D,\delta), \
(FHK)\int_{[a,b]_{\mathbb{T}}}f_{n}(x)\Delta x\right)<\epsilon
$$
for any $D\in \mathfrak{D}(\delta,[a,b]_{\mathbb{T}})$ and for every $n$.
Fix a $D_{0}\in \mathfrak{D}(\delta,[a,b]_{\mathbb{T}})$. From (1) of
Theorem~\ref{thm3.9}, there exists $N$ such that
$$\mathbf{D}\left(S(f_{n},D_{0},\delta),
\ S(f_{m},D_{0},\delta)\right)<\epsilon$$
for arbitray $n, m>N$. Then,
\begin{equation*}
\begin{split}
\mathbf{D}&\left((FHK)\int_{[a,b]_{\mathbb{T}}}f_{n}(x)\Delta x,
\ (FHK)\int_{[a,b]_{\mathbb{T}}}f_{m}(x)\Delta x\right)\\
&\leq \mathbf{D}\left(S(f_{n},D_{0},\delta),
\ (FHK)\int_{[a,b]_{\mathbb{T}}}f_{n}(x)\Delta x\right)
+\mathbf{D}\left(S(f_{n},D_{0},\delta),
\ S(f_{m},D_{0},\delta)\right)\\
&\quad +\mathbf{D}\left(S(f_{m},D_{0},\delta),
\ (FHK)\int_{[a,b]_{\mathbb{T}}}f_{m}(x)\Delta x\right)\\
&< 3\epsilon
\end{split}
\end{equation*}
for any $n, m>N$ and, hence, 
$\left\{(FHK)\int_{[a,b]_{\mathbb{T}}}f_{n}(x)\Delta x\right\}$ 
is a Cauchy sequence. Let
$$
\lim_{n\rightarrow\infty}(FHK)\int_{[a,b]_{\mathbb{T}}}f_{n}(x)\Delta x 
=\tilde{A}.
$$
We now prove that
$$
\tilde{A}=(FHK)\int_{[a,b]_{\mathbb{T}}}f(x)\Delta x.
$$
Let $\epsilon>0$. By hypothesis, there exists a $\delta\in
\Gamma(\Delta, [a,b]_{\mathbb{T}})$ such that
$$
\mathbf{D}\left(S(f_{n},D,\delta),
\ (FHK)\int_{[a,b]_{\mathbb{T}}}f_{n}(x)\Delta x\right)
<\epsilon
$$
for any $D\in \mathfrak{D}(\delta,[a,b]_{\mathbb{T}})$ and for all $n$.
Choose $N$ that satisfies
$$
\mathbf{D}\left((FHK)\displaystyle \int_{[a,b]_{\mathbb{T}}}f_{n}(x)\Delta x,
\ \tilde{A}\right)<\epsilon
$$
for all $n>N$. For the above $D$ and $N$, there exists $N_{0}>N$ satisfying
$$
\mathbf{D}\left(S(f_{N_{0}},D,\delta),
\ S(f,D,\delta)\right)<\epsilon.
$$
Therefore,
\begin{equation*}
\begin{split}
\mathbf{D}&\left(S(f,D,\delta), \ \tilde{A}\right)\\
&\leq \mathbf{D}\left(S(f,D,\delta), \ S(f_{N_{0}},
D,\delta)\right)+\mathbf{D}\left(S(f_{N_{0}},D,\delta),
\ (FHK)\int_{[a,b]_{\mathbb{T}}}f_{N_{0}}(x)\Delta x\right)\\
&\quad +\mathbf{D}\left((FHK)\int_{[a,b]_{\mathbb{T}}}f_{N_{0}}(x)\Delta x, 
\ \tilde{A}\right)\\
&< 3\epsilon
\end{split}
\end{equation*}
and the result follows.
\end{proof}

\begin{definition}[See \cite{C05}]
\label{defn3.4}
A function $f:[a,b]_{\mathbb{T}}\rightarrow \mathbb{R}$ 
is called absolutely continuous on $[a,b]_{\mathbb{T}}$,
if for each $\epsilon>0$ there exists $\gamma>0$ such that
$$
\sum_{i=1}^{n}|f(x_{i})-f(x_{i-1})|<\epsilon
$$
whenever $\bigcup^{n}_{i=1}[x_{i-1},x_{i}]_{\mathbb{T}}\subset [a,b]_{\mathbb{T}}$
and $\sum_{i=1}^{n}\Delta x_{i}<\gamma$.
\end{definition}

\begin{theorem}[Dominated convergence theorem]
\label{thm3.10}
Let the time scale interval $[a,b]_{\mathbb{T}}$ be given. 
If $f_{n}(x)\in\mathcal{FHK}_{[a,b]_{\mathbb{T}}}$, $n=1,2,\ldots$, satisfy
\begin{enumerate}
\item[(1)] $\lim_{n\rightarrow\infty}f_{n}(x)= f(x)$ $\Delta$ a.e.;

\item[(2)] $G(x)\leq f_{n}(t)\leq H(x)$ $\Delta$ a.e. and
$G(x),\ H(x)\in\mathcal{FHK}_{[a,b]_{\mathbb{T}}}$;
\end{enumerate}
then sequence $\{f_{n}(x)\}$ is uniformly FHK $\Delta$-integrable. 
Thus,
$f(x)\in\mathcal{FHK}_{[a,b]_{\mathbb{T}}}$ and
$$
\lim_{n\rightarrow\infty}(FHK)\int_{[a,b]_{\mathbb{T}}}f_{n}(x)\Delta x
=(FHK)\int_{[a,b]_{\mathbb{T}}}f(x)\Delta x.
$$
\end{theorem}

\begin{proof}
By hypothesis, one has
\begin{equation*}
\begin{split}
\mathbf{D}\left(f_{p}(x),f_{q}(x)\right)
&=\sup_{\alpha\in[0,1]}\max\left\{|\underline{f_{p}(x)^{\alpha}}
-\underline{f_{q}(x)^{\alpha}}|,|\overline{f_{p}(x)^{\alpha}}
-\overline{f_{q}(x)^{\alpha}}|\right\}\\
&\leq \sup_{\alpha\in[0,1]}\max\left\{|\underline{H(x)^{\alpha}}
-\underline{G(x)^{\alpha}}|,|\overline{H(x)^{\alpha}}
-\overline{G(x)^{\alpha}}|\right\}\\
&= \mathbf{D}\left(H(x),G(x)\right).
\end{split}
\end{equation*}
Then, $\mathbf{D}\left(H(x),G(x)\right)$ is Lebesgue $\Delta$-integrable.
Let
$$
\mathbf{D}(x)=\int_{[a,x]_{\mathbb{T}}}D\left(H(s),G(s)\right) \Delta s.
$$
From \cite{C05}, $\mathbf{D}(x)$ is absolutely continuous 
on $[a,b]_{\mathbb{T}}$. Let $\epsilon>0$.
Then there exists $\gamma>0$ such that
$$
\sum_{i=1}^{n}|\mathbf{D}(x_{i})-\mathbf{D}(x_{i-1})|<\frac{\epsilon}{b-a}
$$
whenever $\bigcup^{n}_{i=1}[x_{i-1},x_{i}]_{\mathbb{T}}\subset [a,b]_{\mathbb{T}}$
and
$\sum_{i=1}^{n}\Delta x_{i}<\gamma$. The limit
$\lim_{n\rightarrow\infty}f_{n}(x)= f(x)$ holds $\Delta$ a.e. on
$[a,b]_{\mathbb{T}}$ and $\{\mathbf{D}(f_{n}(x), f(x))\}$ is a
sequence of $\Delta$-measurable functions. Thanks to the Egorov's
theorem, there exists an open set $\Omega$ with $m(\Omega)<\delta$ such that
$\lim_{n\rightarrow\infty}f_{n}(x)= f(x)$ uniformly for $x\in
[a,b]_{\mathbb{T}}\backslash \Omega$. Thus, there exists $N$ such that
$\mathbf{D}(f_{p}(x), f_{q}(x))<\frac{\epsilon}{b-a}$ for any $p, q>N$ and for
any $x\in [a,b]_{\mathbb{T}}\backslash \Omega$. Suppose that 
$\delta_{1}\in \Gamma(\Delta, [a,b]_{\mathbb{T}})$ such that
$$
\left|S(\mathbf{D}\left(H(x),G(x)\right),D,\delta_{1})
- \int_{[a,b]_{\mathbb{T}}}\mathbf{D}\left(H(x),G(x)\right)\Delta x\right|
<\epsilon
$$
and
$$
\mathbf{D}\left(S(f_{n},D,\delta_{1}),
\ (FHK)\int_{[a,b]_{\mathbb{T}}}f_{n}(x)\Delta x\right)<\epsilon
$$
for $1\leq n\leq N$ and for any $D\in \mathfrak{D}(\delta_{1},[a,b]_{\mathbb{T}})$. 
Define $\delta\in \Gamma(\Delta, [a,b]_{\mathbb{T}})$ by
\begin{equation*}
\delta(\xi)=
\begin{cases}
\delta_{1}(\xi)& \text{if $\xi\in [a,b]_{\mathbb{T}}\backslash \Omega$},\\
\min\{\delta_{1}(\xi),\rho(\xi,\Omega)\},  & \text{if $\xi\in \Omega$},
\end{cases}
\end{equation*}
where $\rho(\xi,\Omega)=\inf\{|\xi-\xi'|:\xi'\in \Omega\}$. Fix $n>N$. One has
\begin{equation*}
\begin{split}
\mathbf{D}&\left(S(f_{n},D,\delta), S(f_{N},D,\delta)\right)
= \mathbf{D}\left(\sum_{\xi_{i}\in [a,b]_{\mathbb{T}}}f_{n}(\xi_{i})\Delta x_{i},
\sum_{\xi_{i}\in [a,b]_{\mathbb{T}}}f_{N}(\xi_{i})\Delta x_{i}\right)
\end{split}
\end{equation*}
\begin{equation*}
\begin{split}
&\leq \mathbf{D}\left(\sum_{\xi_{i}\in [a,b]_{\mathbb{T}}\backslash \Omega}
f_{n}(\xi_{i})\Delta x_{i},\sum_{\xi_{i}\in [a,b]_{\mathbb{T}}
\backslash \Omega}f_{N}(\xi_{i})\Delta x_{i}\right)\\
&\quad +\mathbf{D}\left(\sum_{\xi_{i}\in \Omega}f_{n}(\xi_{i})\Delta x_{i},
\sum_{\xi_{i}\in \Omega}f_{N}(\xi_{i})\Delta x_{i}\right)\\
&\leq \epsilon +\sum_{\xi_{i}\in \Omega}\mathbf{D}(f_{n}(\xi_{i}),
f_{N}(\xi_{i}))\Delta x_{i}\\
&\leq \epsilon +\left|\sum_{\xi_{i}\in \Omega}\mathbf{D}(H(\xi_{i}),
G(\xi_{i}))\Delta x_{i}-\int_{\Omega}\mathbf{D}(H(x), G(x))\Delta x\right|
+\left|\int_{\Omega}\mathbf{D}(H(x), G(x))\Delta x\right|\\
&\leq 3\epsilon
\end{split}
\end{equation*}
for any $D\in \mathfrak{D}(\delta,[a,b]_{\mathbb{T}})$.
Hence,
\begin{equation*}
\begin{split}
\mathbf{D}&\left(S(f_{n},D,\delta),\ (FHK)\int_{[a,b]_{\mathbb{T}}}f_{n}(x)\Delta x\right)\\
&\leq \mathbf{D}\left(S(f_{n},D,\delta),\ S(f_{N},D,\delta)\right)
+\mathbf{D}\left(S(f_{N},D,\delta),\ (FHK)\int_{[a,b]_{\mathbb{T}}}f_{N}(x)\Delta x\right)\\
&\quad +\mathbf{D}\left((FHK)\int_{[a,b]_{\mathbb{T}}}f_{N}(x)\Delta x,
\ (FHK)\int_{[a,b]_{\mathbb{T}}}f_{n}(x)\Delta x\right)\\
&\leq 5\epsilon.
\end{split}
\end{equation*}
Our dominated convergence theorem is proved.
\end{proof}

As a consequence of Theorem~\ref{thm3.10}, 
we get the following monotone convergence theorem.

\begin{theorem}[Monotone convergence theorem]
\label{thm3.11}
Let the time scale interval $[a,b]_{\mathbb{T}}$ be given. 
If $f_{n}(x)\in\mathcal{FHK}_{[a,b]_{\mathbb{T}}}$,
$n=1,2,\ldots$, satisfy
\begin{enumerate}
\item[(1)] $\lim_{n\rightarrow\infty}f_{n}(x)= f(x)$ $\Delta$ a.e.;

\item[(2)] $\{f_{n}(x)\}$ is a monotone sequence and 
$f_{n}(x)\in\mathcal{FHK}_{[a,b]_{\mathbb{T}}}$;
\end{enumerate}
then $\{f_{n}(x)\}$ is uniformly FHK $\Delta$-integrable. Consequently,
$f(x)\in\mathcal{FHK}_{[a,b]_{\mathbb{T}}}$. Moreover,
$$
\lim_{n\rightarrow\infty}(FHK)\int_{[a,b]_{\mathbb{T}}}f_{n}(x)\Delta x
=(FHK)\int_{[a,b]_{\mathbb{T}}}f(x)\Delta x.
$$
\end{theorem}


\section{Conclusion}
\label{sec:4}

We investigated the fuzzy Henstock--Kurzweil (FHK) 
delta integral on time scales.
Our results give a common generalization of the classical FHK 
and HK integrals. For future researches, we will investigate
the characterization of FHK delta integrable functions.
Another interesting line of research consists to study 
the concept of fuzzy Henstock--Stieltjes integral on time scales.


\begin{acknowledgement}
This research is supported by Chinese Fundamental Research Funds
for the Central Universities, grant 2017B19714 (Ye, Liu and Zhao);
by the Educational Commission of Hubei Province, grant B2016160 (Zhao);
and by Portuguese funds through FCT and CIDMA, within project 
UID/MAT/04106/2013 (Torres).
\end{acknowledgement}




\begin{thebibliography}{xx}

\bibitem{A}
S. Avsec, B. Bannish, B. Johnson, S. Meckler,
\textit{The Henstock-Kurzweil delta integral on unbounded time scales},
PanAmer. Math. J. \textbf{16} (2006), 77--98.

\bibitem{BT16}
B. Bayour, D. F. M. Torres,
\textit{Complex-valued fractional derivatives on time scales},
Differential and difference equations with applications, 79--87,
Springer Proc. Math. Stat., \textbf{164}, Springer, [Cham], (2016).
{\tt arXiv:1511.02153}

\bibitem{BCT2}
N. Benkhettou, A. M. C. Brito da Cruz, D. F. M. Torres,
\textit{A fractional calculus on arbitrary time scales:
fractional differentiation and fractional integration},
Signal Processing \textbf{107} (2015), 230--237.
{\tt arXiv:1405.2813}

\bibitem{BCT1}
N. Benkhettou, A. M. C. Brito da Cruz, D. F. M. Torres,
\textit{Nonsymmetric and symmetric fractional calculus on arbitrary nonempty closed sets},
Math. Methods Appl. Sci. \textbf{39} (2016), 261--279.
{\tt arXiv:1502.07277}

\bibitem{BP1}
M. Bohner, A. Peterson,
\textit{Dynamic equations on time scales: an introduction with applications},
Birkh\"{a}user, Boston, MA (2001).

\bibitem{BP2}
M. Bohner, A. Peterson,
\textit{Advances in dynamic equations on time scales},
Birkh\"{a}user, Boston, MA (2003).

\bibitem{BPM}
B. Bongiorno, L. D. Piazza, K. Musia{\l},
\textit{Kurzweil-Henstock and Kurzweil-Henstock-Pettis integrability of strongly measurable functions},
Math. Bohemica \textbf{131} (2006), 211--223.

\bibitem{BPM12}
B. Bongiorno, L. Di Piazza, K. Musia{\l},
\textit{A decomposition theorem for the fuzzy Henstock integral},
Fuzzy Sets and Systems \textbf{200} (2012), 36--47.

\bibitem{MR3393831}
A. M. C. Brito da Cruz, N. Martins, D. F. M. Torres,
\textit{The diamond integral on time scales},
Bull. Malays. Math. Sci. Soc. \textbf{38} (2015), no.~4, 1453--1462.
{\tt arXiv:1306.0988}

\bibitem{C05}
A. Cabada, D. R. Vivero,
\textit{Criterions for absolute continuity on time scales},
J. Difference Equ. Appl. \textbf{11} (2005), 1013--1028.

\bibitem{C}
M. Cicho\'n,
\textit{On integrals of vector-valued functions on time scales},
Comm. Math. Anal. \textbf{11} (2011), 94--110.

\bibitem{PM}
L. Di Piazza, K. Musia{\l},
\textit{Relations among Henstock, McShane and Pettis integrals
for multifunctions with compact convex values},
Monatsh Math. \textbf{173} (2014), 459--470.

\bibitem{D16}
K. F. Duan,
\textit{The Henstock-Stieltjes integral for
fuzzy-number-valued functions on a infinite interval},
J. Comput. Anal. Appl. \textbf{20} (2016), 928--937.

\bibitem{FB15}
O. S. Fard, T. A. Bidgoli,
\textit{Calculus of fuzzy functions on time scales (I)},
Soft. Comput. \textbf{19} (2015), 293--305.

\bibitem{FB16}
O. S. Fard, D. F. M. Torres, M. R. Zadeh,
\textit{A Hukuhara approach to the study of hybrid fuzzy systems on time scales},
Appl. Anal. Discrete Math. \textbf{10} (2016), 152--167.
{\tt arXiv:1603.03737}

\bibitem{FMS}
M. Federson, J. G. Mesquita, A. Slav\'{\i}k,
\textit{Measure functional differential equations
and functional dynamic equations on time scales},
J. Differential Equations \textbf{252} (2012), 3816--3847.

\bibitem{GV86}
R. Goetschel, W. Voxman,
\textit{Elementary fuzzy calculus},
Fuzzy Sets and Systems \textbf{18} (1986), 31--43.

\bibitem{G04}
Z. T. Gong,
\textit{The convergence theorems of the McShane interal of  fuzzy-valued functions},
Southeast Asian Bull. Math. \textbf{27} (2003), 55--62.

\bibitem{GS09}
Z. T. Gong, Y. B. Shao,
\textit{The controlled convergence theorems for the strong
Henstock integrals of fuzzy-number-valued functions},
Fuzzy Sets and Systems \textbf{160} (2009), 1528--1546.

\bibitem{G}
R. A. Gordon,
\textit{The integrals of Lebesgue, Denjoy, Perron, and Henstock},
Graduate Studies in Mathematics, 4,
American Mathematical Society, Providence, RI (1994).

\bibitem{G1}
G. Sh. Guseinov,
\textit{Integration on time scales},
J. Math. Anal. Appl. \textbf{285} (2003), 107--127.

\bibitem{HG15}
M. E. Hamid, Z. T. Gong,
\textit{The characterizations of McShane integral and Henstock integrals
for fuzzy-number-valued functions with a small Riemann sum on a small set},
J. Comput. Anal. Appl. \textbf{19} (2015), 830--836.

\bibitem{H1}
R. Henstock,
\textit{Definitions of Riemann Type of Variational Integral},
Proc. London Math. Soc. \textbf{11} (1961), 402--418.

\bibitem{H4}
S. Hilger,
\textit{E$\imath$n Ma{\ss}kettenkalk\"{u}l m$\imath$t
Anwendung auf Zentrumsmann$\imath$gfalt$\imath$gke$\imath$ten},
Ph.D. Thesis, Universt\"{a}t W\"{u}rzburg (1988).

\bibitem{K87}
O. Kaleva,
\textit{Fuzzy differential equations},
Fuzzy Sets and Systems \textbf{24} (1987), 301--317.

\bibitem{K1}
J. Kurzweil,
\textit{Generalized ordinary differential equations and continuous dependence on a parameter},
Czech. Math. J. \textbf{7} (1957), 418--446.

\bibitem{L2}
T. Y. Lee,
\textit{Henstock-Kurzweil integration on Euclidean spaces},
Series in Real Analysis, 12,
World Scientific Publishing Co. Pte. Ltd., Hackensack, NJ (2011).

\bibitem{MS2}
G. A. Monteiro, A. Slav\'{\i}k,
\textit{Generalized elementary functions},
J. Math. Anal. Appl. \textbf{411} (2014), 838--852.

\bibitem{MS1}
G. A. Monteiro, A. Slav\'{\i}k,
\textit{Extremal solutions of measure differential equations},
J. Math. Anal. Appl. \textbf{444} (2016), 568--597.

\bibitem{M15}
K. Musia{\l},
\textit{A decomposition theorem for Banach space valued fuzzy Henstock integral},
  Fuzzy Sets and Systems \textbf{259} (2015), 21--28.

\bibitem{OTT}
M. D. Ortigueira, D. F. M. Torres, J. J. Trujillo,
\textit{Exponentials and Laplace transforms on nonuniform time scales},
Commun. Nonlinear Sci. Numer. Simul. \textbf{39} (2016), 252--270.
{\tt arXiv:1603.04410}

\bibitem{PT}
A. Peterson, B. Thompson,
\textit{Henstock-Kurzweil delta and nabla integrals},
J. Math. Anal. Appl. \textbf{323} (2006), 162--178.

\bibitem{ST1}
B. R. Satco, C. O. Turcu,
\textit{Henstock-Kurzweil-Pettis integral and weak topologies
in nonlinear integral equations on time scales},
Math. Slovaca \textbf{63} (2013), 1347--1360.

\bibitem{SY}
\v{S}. Schwabik, G. J. Ye,
\textit{Topics in Banach space integration},
Series in Real Analysis, 10,
World Scientific Publishing Co. Pte. Ltd., Hackensack, NJ (2005).

\bibitem{SH13}
Y. B. Shao, H. H. Zhang,
\textit{The strong fuzzy Henstock integrals and discontinuous fuzzy differential equations},
J. Appl. Math. 2013, Art. ID 419701, 8~pp.

\bibitem{SH014}
Y. B. Shao, H. H. Zhang,
\textit{Existence of the solution for discontinuous fuzzy
integro-differential equations and strong fuzzy Henstock integrals},
Nonlinear Dyn. Syst. Theory \textbf{14} (2014), 148--161.

\bibitem{S3}
A. Slav\'{\i}k,
\textit{Generalized differential equations: differentiability
of solutions with respect to initial conditions and parameters},
J. Math. Anal. Appl. \textbf{402} (2013), 261--274.

\bibitem{S1}
A. Slav\'{\i}k,
\textit{Kurzweil and McShane product integration in Banach algebras},
J. Math. Anal. Appl. \textbf{424} (2015), 748--773.

\bibitem{S2}
A. Slav\'{\i}k,
\textit{Well-posedness results for abstract generalized differential
equations and measure functional differential equations},
J. Differential Equations \textbf{259} (2015), 666--707.

\bibitem{T}
B. S. Thompson,
\textit{Henstock-Kurzweil integrals on time scales},
Panamer. Math. J. \textbf{18} (2008), 1--19.

\bibitem{WG00}
C. X. Wu, Z. T. Gong,
\textit{On Henstock integrals of interval-valued functions and fuzzy-valued functions},
Fuzzy Sets and Systems \textbf{115} (2000), 377--391.

\bibitem{WG01}
C. X. Wu, Z. T. Gong,
\textit{On Henstock integral of fuzzy-number-valued functions(I)},
Fuzzy Sets and Systems \textbf{120} (2001), 523--532.

\bibitem{Y}
G. J. Ye,
\textit{On Henstock-Kurzweil and McShane integrals of Banach space-valued functions},
J. Math. Anal. Appl. \textbf{330} (2007), 753--765.

\bibitem{MyID:375}
X. X. You, D. F. Zhao, D. F. M. Torres,
\textit{On the Henstock--Kurzweil integral for Riesz-space-valued functions on time scales},
J. Nonlinear Sci. Appl. \textbf{10} (2017), 2487--2500.
{\tt arXiv:1704.06808}

\bibitem{ZY}
D. F. Zhao, G. J. Ye,
\textit{On ap-Henstock-Stieltjes integral},
J. Chungcheong Math. Soc. \textbf{19} (2006), 177--187.

\end{thebibliography}
\end{document}